\documentclass[12pt]{article}
\usepackage{etex}
\usepackage[a4paper]{geometry}
\usepackage{amsmath,amsfonts,amssymb,amsthm,mathrsfs,bm}
\usepackage[english]{babel}

\usepackage[usenames,dvipsnames]{xcolor}
\usepackage{tikz}
\usepackage[all]{xy}
\usepackage{graphicx}
\usepackage{mathrsfs}
\usepackage{pstricks}
\usepackage{pstricks-add}
\psset{unit=1mm}
\usepackage{enumitem}
\usepackage{color, colortbl, xcolor}
\setlist{nolistsep}

\usepackage[pdfauthor={S. Kamada, V. Lebed, K. Tanaka},
pdftitle={The shadow nature of positive and twisted quandle cocycle invariants of knots},%
colorlinks=false,linkbordercolor=Gray,citebordercolor=Gray,urlbordercolor=Gray]{hyperref}
\usepackage{caption}
\usepackage[all]{hypcap}

  \theoremstyle{plain}
\newtheorem{theorem}{Theorem}
\newtheorem{proposition}{Proposition}[section]

  \theoremstyle{remark}
\newtheorem{remark}[proposition]{Remark}
  \theoremstyle{definition}
\newtheorem{definition}[proposition]{Definition}

\newtheorem{lemma}[proposition]{Lemma}

\definecolor{MyGrey}{rgb}{.704,.704,.646}
\definecolor{MyGreyBis}{rgb}{.804,.804,.756}
\definecolor{MyBlue}{rgb}{.14,.355,.704} 
\definecolor{MyViolet}{rgb}{.545,.1,.5} 
\definecolor{MyRed}{rgb}{0.81,0,0}
\definecolor{MyGreen}{rgb}{.545,.545,0}

\newcommand*{\wop}{\mathrel{\ensuremath{\widetilde{\lhd}}}}
\newcommand*{\op}{\mathrel{\ensuremath{\lhd}}}
\newcommand*{\opp}{l}
\newcommand*{\opr}{r}

\newcommand*{\ZZ}{\mathbb{Z}}

\newcommand*{\RR}{\mathbb{R}}

\newcommand*\Quandle{{\scriptscriptstyle\mathrm{Q}}}
\newcommand*{\Q}{Q}
\newcommand*{\M}{M}
\newcommand*{\D}{D}
\newcommand*{\A}{Arcs}
\newcommand*{\Reg}{Regions}
\newcommand*{\C}{\mathcal C}
\newcommand*{\Csh}{\mathcal C^{sh}}
\newcommand*{\Corb}{\mathcal C_{*}}
\renewcommand*{\Col}{\mathscr C}
\newcommand*{\Colsh}{\mathscr C^{sh}}
\newcommand*{\w}{\omega}
\newcommand*{\wsh}{\omega_{\alpha}}
\newcommand*{\BW}{\mathcal{W}}
\newcommand*{\BWsh}{\mathcal{W}^{sh}}
\newcommand*{\BWpos}{\mathcal{W}^{pos}}
\newcommand*{\BWtw}{\mathcal{W}^{tw}}
\newcommand*{\BWtwsh}{\mathcal{W}^{tw,sh}}
\newcommand*{\x}{\psset{unit=0.3mm}\begin{pspicture}(10,10)
\psline(0,10)(10,0)
\psline[border=0.8,arrowsize=3]{->}(10,10)(0,0)
\end{pspicture}\psset{unit=1mm}}
\newcommand*{\xx}{\psset{unit=0.2mm}\begin{pspicture}(10,10)
\psline[linewidth=0.6](0,10)(10,0)
\psline[border=0.8,arrowsize=2.5,linewidth=0.6]{->}(10,10)(0,0)
\end{pspicture}\psset{unit=1mm}}
\newcommand*{\sgn}{\varepsilon}
\newcommand*{\sgnpos}{\varepsilon^{pos}}
\newcommand*{\ind}{i}
\renewcommand{\O}{\mathcal O}
\newcommand{\Orb}{{Orb}}
\newcommand{\oalpha}{\overline{\alpha}}

\newcommand*\rcircled[1]{$\,$\tikz[baseline=(char.base)]{
            \node[shape=circle,draw,inner sep=1pt] (char) {#1};}}

\begin{document}

\title{The shadow nature of positive and twisted \\ quandle cocycle invariants of knots}

\author{
  Seiichi Kamada\\
  \texttt{skamada@sci.osaka-cu.ac.jp}\\
  Osaka City University
  \and
  Victoria Lebed\\ 
  \texttt{lebed.victoria@gmail.com}\\
  OCAMI, Osaka City University
  \and
  Kokoro Tanaka\\ 
  \texttt{kotanaka@u-gakugei.ac.jp}\\
  Tokyo Gakugei University}

\maketitle

\begin{abstract}
\footnotesize 
Quandle cocycle invariants form a powerful and well developed tool in knot theory. This paper treats their variations --- namely,  positive and twisted quandle cocycle invariants, and shadow invariants. We interpret the former as particular cases of the latter. As an application, several constructions from the shadow world are extended to the positive and twisted cases. Another application is a sharpening of twisted quandle cocycle invariants for multi-component links.
\end{abstract}


\section{Introduction}\label{S:intro}

A \emph{quandle} is a set~$\Q$ endowed with binary operations\footnote{Operation~$\wop$ is deduced from~$\op$ using~\eqref{E:Inv}; we thus often simply speak about a quandle $(\Q,\op)$.} $\op$ and~$\wop$ satisfying 
\begin{align}
&(a \op b)\op c = (a \op c)\op(b \op c), \label{E:SD}\\
&(a \op b) \wop b = (a \wop b) \op b = a,\label{E:Inv}\\
&a \op a = a. \label{E:Idem}
\end{align}
A group with the conjugation operation $a \op b = b^{-1}ab$ is an important example. \linebreak%
 Quandles were introduced in~\cite{Joyce,Matveev} as an efficient tool for constructing knot\footnote{In this paper by \emph{knots} we always mean \emph{oriented knots and links in~$\RR^3$}.} invariants. Concretely, for a quandle~$\Q$ and a knot diagram~$\D$, consider the set $\Col_{\Q}(\D)$ of \emph{$\Q$-colorings} of~$\D$ --- that is, assignments $\C\colon \A(\D) \to \Q$ of elements of~$\Q$ to every arc of~$\D$, satisfying the condition on Figure~\ref{pic:Colorings}\rcircled{A} around each crossing.\footnote{Here and afterwards diagrams with unoriented arcs mean that all coherently oriented versions of these diagrams should be considered.} Axioms \eqref{E:SD}-\eqref{E:Idem} force the number $|\Col_{\Q}(\D)|$ of such colorings to be stable under Reidemeister moves, and thus to give a knot invariant. 
\begin{center}
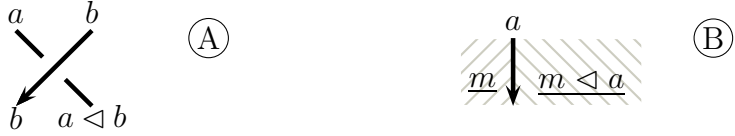

\begin{pspicture}(0,-2)(60,12)
\psline[linewidth=0.6](0,10)(10,0)
\psline[linewidth=0.6,border=1.8,arrowsize=2]{->}(10,10)(0,0)
\rput[b](0,11){$a$}
\rput[b](10,11){$b$}
\rput[b](0,-3){$b$}
\rput[b](10,-3){$a \lhd b$}
\rput(25,9){\rcircled{A}}
\end{pspicture}
\begin{pspicture}(0,-2)(35,12)
\psframe[fillstyle=hlines,hatchcolor=MyGreyBis,linestyle=none](-3,0)(4,9)
\psframe[fillstyle=vlines,hatchcolor=MyGreyBis,linestyle=none](4,0)(21,9)
\psline[linewidth=0.6,arrowsize=2]{->}(4,9)(4,0)
\rput(4,11){$a$}
\rput(0,3){$\underline{m}$}
\rput(13,3){$\underline{m \op a}$}
\rput(30,9){\rcircled{B}}
\end{pspicture}
\captionof{figure}{Quandle and shadow colorings}
\label{pic:Colorings}
\end{center} 
 
These \emph{quandle invariants} are very powerful, but they do not detect, for instance, the chirality of knots. This flaw was fixed in~\cite{QuandleHom} by enriching colorings with \emph{weights}. Concretely, given an abelian group~$A$, a map $\w\colon\Q \times \Q \to A$ satisfying
\begin{align}
&\w(a, b) +\w(a \op b, c) = \w(a \op c, b \op c) + \w(a, c), \label{E:BWR3}\\
&\w(a, a)= 0 \label{E:BWR1}
\end{align}
is called a \emph{quandle $2$-cocycle} of~$\Q$. (The word ``cocycle'' will be justified below.) The \emph{$\w$-weight} of a $\Q$-colored knot diagram $(\D,\C)$ is defined as a sum
\psset{unit=0.4mm}
\begin{align}
\BW_{\w}(\D,\C) &= \sum_{\begin{pspicture}(13,10)(-3,0)
\psline(0,10)(10,0)
\psline[border=0.8,arrowsize=3]{->}(10,10)(0,0)
\rput(-2,10){$\scriptstyle a$}
\rput(12,10){$\scriptstyle b$}
\end{pspicture}} \sgn(\x) \w(a, b)\label{E:TotalWeight}
\end{align}
\psset{unit=1mm}
over all crossings~$\x$ of~$\D$, where $\sgn(\x)$ is the \emph{sign} of~$\x$ (cf. Figure~\ref{pic:CocInvar}). Axioms \eqref{E:BWR3}-\eqref{E:BWR1} were chosen so that the the multiset of $\w$-weights $\{\,\BW_{\w}(\D,\C) \,|\, \C \in \Col_{\Q}(\D)\,\}$ yields a knot invariant, called a \emph{quandle cocycle invariant}.

\begin{center}
\begin{pspicture}(-25,-2)(50,13)
\rput(-18,5){$\boxed{\sgn(\x)}$}
\psline[linewidth=0.6,arrowsize=2]{->}(0,10)(10,0)
\psline[linewidth=0.6,border=1.8,arrowsize=2]{->}(10,10)(0,0)
\rput[b](0,11){$a$}
\rput[b](10,11){$b$}
\rput[b](0,-3){$b$}
\rput[b](10,-3){$a \lhd b$}
\rput(30,5){$\mapsto {\color{MyRed}\bm{+}}\w(a,b)$}
\end{pspicture}
\begin{pspicture}(0,-2)(46,12)
\psline[linewidth=0.6,arrowsize=2]{->}(10,10)(0,0)
\psline[linewidth=0.6,border=1.8,arrowsize=2]{->}(0,10)(10,0)
\rput[b](0,11){$b$}
\rput[b](10,11){$a \lhd b$}
\rput[b](0,-3){$a$}
\rput[b](10,-3){$b$}
\rput(30,5){$\mapsto {\color{MyRed}\bm{-}}\w(a,b)$}
\end{pspicture}
\captionof{figure}{Quandle cocycle invariants}
\label{pic:CocInvar}
\end{center}

This construction admits several variations. Instead of quandle $2$-cocycles, they use some other cohomological data, which we now recall. Together with a quandle~$\Q$, consider a \emph{$\Q$-module}\footnote{$\Q$-modules appeared under the name \emph{$\Q$-(quandle-)sets} (\cite{RackHom}), and are also known as \emph{$\Q$-shadows} (\cite{ChangNelson}).} --- that is, a set~$\M$ endowed with two maps $\op,\wop\colon \M \times \Q \to \M$\footnote{The notation is slightly abusive but convenient.} satisfying Axioms \eqref{E:SD}-\eqref{E:Inv} for all $a \in \M$, $b,c \in \Q$. The simplest examples are $\Q$ itself, and a one-element set (called a \textit{trivial} $\Q$-module), with evident module operations. Take also an abelian group~$A$. Denote by $C^k(\M,\Q,A)$ the abelian group of maps from~$\M \times \Q^{\times k}$ to~$A$, and put\footnote{Subscripts~$\opp$ \& $\opr$ refer to an interpretation of these maps as \emph{left} \& \emph{right} differentials, cf.~\cite{Lebed1}.}
\begin{align}
(d^k_{\opp} \phi)(m,a_1, \ldots, a_{k+1}) &= \sum_{i=1}^{k+1} (-1)^{i-1} \phi(m \op a_i,\ldots,a_{i-1} \op a_i, a_{i+1}, \ldots,a_{k+1}), \label{E:RackCohomL}\\
(d^k_{\opr} \phi)(m,a_1, \ldots, a_{k+1}) &= \sum_{i=1}^{k+1} (-1)^{i-1} \phi(m, a_1,\ldots,a_{i-1},a_{i+1}, \ldots,a_{k+1}). \label{E:RackCohomR}
\end{align}
Maps~$d^k_{\opp}$ and~$d^k_{\opr}$ turn out to be anti-commuting differentials. Hence for all $\alpha_{\opp}, \alpha_{\opr}$ lying in~$\ZZ$ (or, more generally, in a commutative ring~$R$ if $A$ is an $R$-module), $(C^k(\M,\Q,A),\alpha_{\opp} d^k_{\opp} - \alpha_{\opr} d^k_{\opr})$ is a cochain complex. Denote by $C_{\Quandle}^k(\M,\Q,A)$ the subgroup of maps vanishing on all tuples $(m,a_1, \ldots, a_k)$ with $a_i = a_{i+1}$ for some~$i$. Both $d^k_{\opp}$ and $d^k_{\opr}$ restrict to $C_{\Quandle}^k(\M,\Q,A)$, and thus so do their linear combinations. The \emph{cocycles~/ coboundaries~/ cohomology groups} of $(C_{\Quandle}^k(\M,\Q,A),\alpha_{\opp} d^k_{\opp} - \alpha_{\opr} d^k_{\opr})$ receive the adjective
\begin{itemize}
\item \emph{quandle} if $\alpha_{\opp} = \alpha_{\opr} =1$ (\cite{RackHom,QuandleHom}); 
\item \emph{positive quandle} if $\alpha_{\opp} = -\alpha_{\opr} =1$ (\cite{PosQuandleHom}); 
\item \emph{($t$-)twisted quandle} if $\alpha_{\opp} =1$,  $\alpha_{\opr} =t$, and $A$ is a $\ZZ[t^{\pm 1}]$-module (\cite{TwistedQuandle}); 
\item \emph{two-term distributive} in the general case (\cite{PrzSikora,Prz1}). 
\end{itemize} 
Note that a trivial~$\M$ can be safely excluded from consideration (and from our notations), since $\M \times \Q^{\times k} \cong \Q^{\times k}$; we will often talk about \emph{($\M$-)shadow} cocycles~/ coboundaries~/ cohomology groups if $\M$ is non-trivial.

\emph{Positive} (\cite{PosQuandleHom}) and \emph{twisted} (\cite{TwistedQuandle}) variations of quandle cocycle invariants using positive and, respectively, twisted quandle $2$-cocycles with trivial~$\M$ are recalled in Section~\ref{S:Twisted}. \emph{Shadow invariants} based on shadow quandle $2$-cocycles are also reviewed; together with arc colorings, these use region colorings by elements of~$\M$, respecting the rule on Figure~\ref{pic:Colorings}\rcircled{B}. In literature, mostly the case $\M=\Q$ was considered formally and thoroughly (\cite{RS_Links,CKS_Geometric,Kamada}); the general case seems to be folklore. In this paper, carefully chosen $\Q$-modules allow us to present twisted (Section~\ref{S:TwistedShadow}) and positive (Section~\ref{S:PosShadow}) quandle cocycle invariants as shadow invariants. This instantly gives a proof of their invariance, avoiding technical verifications. Moreover, all results established for shadow invariants can now be applied in the positive and twisted cases. Such results, discussed in Section~\ref{S:Applications}, include the equality of invariants obtained from cohomologous cocycles; a certain symmetry of the invariants; some restrictions on the values of weights one can encounter; higher-dimensional generalizations; and shadow\footnote{Note the ``double-shadow'' situation here.} versions. In Section~\ref{S:TwistedLinks}, the shadow ideas lead to a sharpening of twisted quandle cocycle invariants for multi-component links.

\section{Shadow, positive, and twisted quandle cocycle invariants}\label{S:Twisted}

In this section we recall three variations of the quandle cocycle invariant construction. From now on, fix a quandle $(\Q,\op)$ and an abelian group~$A$.

\medskip
First, together with the arcs of a knot diagram~$\D$, one can color its \emph{regions} (= the connected components of $\RR^2\backslash\D$) by elements of a $\Q$-module $\M$. Concretely, define a \emph{(shadow) $(\Q,\M)$-coloring} of~$\D$ as a $\Q$-coloring $\C$ of the arcs of~$\D$ and an assignment $\Csh\colon \Reg(\D) \to \M$, compatible in the sense of Figure~\ref{pic:Colorings}\rcircled{B}\footnote{In our diagrams, region colors are underlined for a better readability.}. The set of such colorings is denoted by $\Colsh_{\Q,\M}(\D)$. The color of the exterior region $\Delta_{ex}(\D)$ of~$\D$ uniquely determines all the other region colors, hence the $(\Q,\M)$-coloring number $|\Colsh_{\Q,\M}(\D)|$ equals $|\Col_{\Q}(\D)|\cdot |\M|$ and says nothing new about the knot. The situation changes when weights enter the story. Take an \emph{$\M$-shadow quandle $2$-cocycle} of~$\Q$ --- that is, a map $\w\colon \M \times \Q \times \Q \to A$ satisfying
\begin{align}
\w(m, a, b)  &+ \w(m \op b, a \op b, c) + \w(m, b, c) = \notag\\
&\w(m \op c, a \op c, b \op c) + \w(m, a, c) + \w(m \op a, b, c), \label{E:BWR3sh}\\
\w(m, a, a) &= 0. \label{E:BWR1sh}
\end{align}

\begin{definition}
The \emph{(shadow) $\w$-weight} of a $(\Q,\M)$-colored knot diagram $(\D,\C,\Csh)$ is the sum
\psset{unit=0.4mm}
\begin{align}
\BWsh_{\w}(\D,\C,\Csh) &= \sum_{\begin{pspicture}(13,10)(-8,0)
\psline[arrowsize=3]{->}(0,10)(10,0)
\psline[border=0.8,arrowsize=3]{->}(10,10)(0,0)
\rput(-2,10){$\scriptstyle a$}
\rput(12,10){$\scriptstyle b$}
\rput(-6,4){$\scriptstyle \underline{m}$}
\end{pspicture}}\w(m, a, b) 
-\sum_{\begin{pspicture}(13,10)(-8,0)
\psline[arrowsize=3]{->}(10,10)(0,0)
\psline[border=0.8,arrowsize=3]{->}(0,10)(10,0)
\rput(-2,0){$\scriptstyle a$}
\rput(12,0){$\scriptstyle b$}
\rput(-6,6){$\scriptstyle \underline{m}$}
\end{pspicture}}\w(m, a, b)\label{E:TotalWeightSh}
\end{align}
\psset{unit=1mm}
taken over all crossings of~$\D$.
\end{definition}

\begin{theorem}\label{T:ShadowInv}
For any $\Q$-module~$\M$, $m \in \M$, and $\M$-shadow quandle $2$-cocycle $\w$ of~$\Q$, the multiset 
$$\{\,\BWsh_{\w}(\D,\C,\Csh) \,|\,(\C,\Csh) \in \Colsh_{\Q,\M}(\D),\, \Csh(\Delta_{ex}(\D))=m \,\}$$ 
depends only on the underlying knot, and not on the choice of its diagram~$\D$. 
\end{theorem}

The (extremely rich) knot invariant thus obtained are called \emph{shadow quandle cocycle invariants}, or simply \emph{shadow invariants}.

\medskip
Another variation uses a \emph{positive quandle $2$-cocycle} of~$\Q$ --- that is, a map $\w\colon\Q \times \Q \to A$ satisfying~\eqref{E:BWR1} and
\begin{align}
&\w(a, c) +\w(a \op b, c) = \w(a \op c, b \op c) + \w(a, b) + 2\w(b, c). \label{E:BWR3pos}
\end{align}

\begin{definition}
The \emph{(positive) $\w$-weight} of a $\Q$-colored knot diagram $(\D,\C)$ is the sum
\psset{unit=0.4mm}
\begin{align}
\BWpos_{\w}(\D,\C) &= \sum_{\begin{pspicture}(13,10)(-3,0)
\psline(0,10)(10,0)
\psline[border=0.8,arrowsize=3]{->}(10,10)(0,0)
\rput(-2,10){$\scriptstyle a$}
\rput(12,10){$\scriptstyle b$}
\end{pspicture}} \sgnpos(\x) \w(a, b),\label{E:TotalWeightPos}
\end{align}
\psset{unit=1mm}
where the sign $\sgnpos(\x)$ is defined on Figure~\ref{pic:ColoringsPos} (here the checkerboard coloring of the diagram's regions is used, with the exterior region declared white). 
\end{definition}

\begin{center}
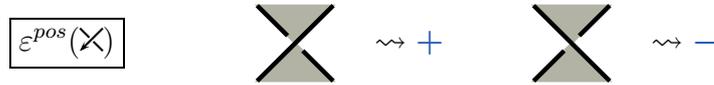

\begin{pspicture}(-30,0)(35,10)
\rput(-25,5){$\boxed{\sgnpos(\x)}$}
\pspolygon[fillstyle=solid,fillcolor=MyGrey,linecolor=MyGrey](0,0)(10,0)(5,5)
\pspolygon[fillstyle=solid,fillcolor=MyGrey,linecolor=MyGrey](0,10)(10,10)(5,5)
\psline[linewidth=0.6](10,10)(0,0)
\psline[linewidth=0.6](0,10)(4,6)
\psline[linewidth=0.6](6,4)(10,0)
\rput(20,5){$\leadsto \color{MyBlue} \bm{+}$}
\end{pspicture}
\begin{pspicture}(0,0)(30,10)
\pspolygon[fillstyle=solid,fillcolor=MyGrey,linecolor=MyGrey](0,0)(10,0)(5,5)
\pspolygon[fillstyle=solid,fillcolor=MyGrey,linecolor=MyGrey](0,10)(10,10)(5,5)
\psline[linewidth=0.6](0,10)(10,0)
\psline[linewidth=0.6](10,10)(6,6)
\psline[linewidth=0.6](4,4)(0,0)
\rput(20,5){$\leadsto \color{MyBlue} \bm{-}$}
\end{pspicture}
\captionof{figure}{Positive quandle cocycle invariants}
\label{pic:ColoringsPos}
\end{center}

\begin{theorem}[\cite{PosQuandleHom}]\label{T:PosInv}
For any positive quandle $2$-cocycle $\w$ of~$\Q$, the multiset $\{\,\BWpos_{\w}(\D,\C) \,|\, \C \in \Col_{\Q}(\D) \,\}$ yields a knot invariant. 
\end{theorem}

The knot invariant obtained is referred to as \emph{positive quandle cocycle invariant}.

\medskip
In order to present the third generalization of quandle cocycle invariants we are interested in,
some technical but intuitive definitions are necessary.
\begin{definition}\label{D:IndexRegion}
\begin{itemize}
\item The \emph{index}\footnote{The term \emph{Alexander number} is sometimes used instead (\cite{AlexanderNb,CS_surfaces,CKS_AlexanderNb}).} $\ind(\Delta)$ of a region~$\Delta$ in a knot diagram $\D$ is the algebraic intersection number for~$\D$ and a path~$\gamma$ connecting $\Delta$ to the exterior region $\Delta_{ex}(\D)$. Sign conventions are indicated in Figure~\ref{pic:Index}\rcircled{A}; $\gamma$ and~$\D$ are supposed to have a finite number of simple transverse intersections. 

\item The \emph{source region} $\Delta(\x)$ of a crossing $\x$ in~$\D$ is defined in Figure~\ref{pic:Index}\rcircled{B}\footnote{In our diagrams, a solid crossing stands for a crossing of arbitrary sign.}\footnote{In other words, it is the region from which all normals to the arcs of~$\x$ point, hence the term.}. 

\item The \emph{index} of a crossing~$\x$ is set to be $\ind(\x) = \ind(\Delta(\x))$. 
 \end{itemize}
\end{definition}

The index of a region/crossing does not depend on the choice of path~$\gamma$ satisfying the description above, and is thus well defined.

\begin{center}
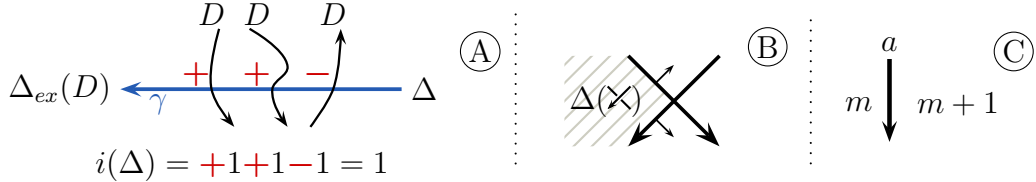

\begin{pspicture}(-10,-5)(55,15)
\psline[linewidth=0.5,linecolor=MyBlue,arrowsize=2]{->}(35,5)(-2,5)
\rput(38,5){$\Delta$}
\rput(-10,5){$\Delta_{ex}(\D)$}
\rput(3,3){$\color{MyBlue} \gamma$}
\pscurve[linewidth=0.3,arrowsize=1.5]{->}(11,13)(10,5)(13,0)
\rput(10,15){$\D$}
\rput(8,7){$\color{MyRed} \bm{+}$}
\pscurve[linewidth=0.3,arrowsize=1.5]{->}(15,13)(20,8)(18,5)(20,0)
\rput(16,15){$\D$}
\rput(16,7){$\color{MyRed} \bm{+}$}
\pscurve[linewidth=0.3,arrowsize=1.5]{->}(23,0)(26,5)(27,13)
\rput(26,15){$\D$}
\rput(24,7){$\color{MyRed} \bm{-}$}
\rput(14,-5){${\ind(\Delta) = {\color{MyRed} \bm{+}}1 {\color{MyRed} \bm{+}}1 {\color{MyRed} \bm{-}}1 = 1}$}
\rput(45,10){\rcircled{A}}
\psline[linestyle=dotted](50,-5)(50,15)
\end{pspicture}
\psset{unit=1.2mm}
\begin{pspicture}(-7,-2)(20,12)
\pspolygon[fillstyle=hlines,hatchcolor=MyGreyBis,linestyle=none](0,0)(-7,0)(-7,10)(0,10)(5,5)
\psline[linewidth=0.4,arrowsize=2]{->}(10,10)(0,0)
\psline[linewidth=0.4,arrowsize=2]{->}(0,10)(10,0)
\psline[linewidth=0.2,arrowsize=0.8]{->}(3,3)(5,1)
\psline[linewidth=0.2,arrowsize=0.8]{->}(3,7)(5,9)
\rput(-2.4,5){$\Delta(\x)$}
\rput(15,11){\rcircled{B}}
\psline[linestyle=dotted](20,-2)(20,14)
\end{pspicture}
\psset{unit=1mm}
\begin{pspicture}(-5,-5)(20,15)
\psline[linewidth=0.6,arrowsize=2]{->}(4,9)(4,-2)
\rput(4,11){$a$}
\rput(0,3){${m}$}
\rput(13,3){${m +1}$}
\rput(20,10){\rcircled{C}}
\end{pspicture}
\captionof{figure}{Regions and indices}\label{pic:Index}
\end{center}

Until the end of Section~\ref{S:Applications}, suppose~$A$ to be a module over a commutative ring~$R$, and fix an $\alpha \in R^*$. Take an \emph{$\alpha$-twisted quandle $2$-cocycle} of~$\Q$ --- that is, a map $\w\colon\Q \times \Q \to A$ satisfying~\eqref{E:BWR1} and
\begin{align}
&\w(a \op c, b \op c) - \alpha \w(a, b) - \w(a \op b, c) + \alpha \w(a, c) + (1 - \alpha)\w(b, c) =0. \label{E:BWR3multi}
\end{align}

\begin{definition}
The \emph{(twisted, or $\alpha$-twisted) $\w$-weight} of a $\Q$-colored knot diagram $(\D,\C)$ is the sum
\psset{unit=0.4mm}
\begin{align}
\BWtw_{\w}(\D,\C,\alpha) &= \sum_{\begin{pspicture}(13,10)(-3,0)
\psline(0,10)(10,0)
\psline[border=0.8,arrowsize=3]{->}(10,10)(0,0)
\rput(-2,10){$\scriptstyle a$}
\rput(12,10){$\scriptstyle b$}
\end{pspicture}} \sgn(\x) \alpha^{-\ind(\x)} \w(a, b)\label{E:TotalWeightMulti}
\end{align}
\psset{unit=1mm}
taken over all crossings of~$\D$. (Recall that $\sgn(\x)$ is the usual sign of~$\x$, cf. Figure~\ref{pic:CocInvar}.)
\end{definition}

\begin{theorem}[\cite{TwistedQuandle}]\label{T:Multi}
For any $\alpha \in R^*$ and any $\alpha$-twisted quandle $2$-cocycle $\w$ of~$\Q$, the multiset $\{\,\BWtw_{\w}(\D,\C,\alpha) \,|\, \C \in \Col_{\Q}(\D)\,\}$ yields a knot invariant.
\end{theorem}

This knot invariant is called \emph{($\alpha$-)twisted quandle cocycle invariant}.

\medskip
To prove Theorems~\ref{T:PosInv} and~\ref{T:Multi}, one could check directly that Reidemeister moves and induced coloring changes preserve the positive/twisted $\w$-weight, as it was done in~\cite{PosQuandleHom,TwistedQuandle}.  Here we prefer a more conceptual solution: in the following sections, the multisets from these theorems are presented as shadow invariants. 

\section{Twisted quandle cocycle invariants via shadows}\label{S:TwistedShadow}

The first ingredient we need is the $\Q$-module $\ZZ$ with $m \op a =m+1$. For this $\Q$-module, the region coloring rule from Figure~\ref{pic:Colorings}\rcircled{B} specializes to the one from Figure~\ref{pic:Index}\rcircled{C}. Comparing it with the definition of the index of a region, one proves

\begin{lemma}\label{L:IndColoring}
Consider a $(\Q,\ZZ)$-coloring of a knot diagram~$\D$ such that the exterior region receives color~$0$. Then the color of any region~$\Delta$ coincides with~$\ind(\Delta)$. 
\end{lemma}

The region coloring from the lemma will be denoted by~$\C^{ind}\colon\Reg(\D) \to \ZZ$. Remark that it is compatible with any arc coloring.

Next, we show how to transform a twisted quandle $2$-cocycle into a (non-twisted) quandle $2$-cocycle, at the cost of introducing our non-trivial $\Q$-module $\ZZ$.
\begin{lemma}\label{L:CocyclesUnitaryShadow}
A map $\w \colon \Q \times \Q \to A$ is an $\alpha$-twisted quandle $2$-cocycle if and only if the map
\begin{align}
\wsh\colon \ZZ \times \Q \times \Q &\longrightarrow A,\notag\\
(m,a,b) &\longmapsto \alpha^{-m} \w(a,b)\label{E:AlphaWeight}
\end{align}
is a shadow quandle $2$-cocycle.
\end{lemma}

\begin{proof}
Conditions~\eqref{E:BWR1} for~$\w$ and~\eqref{E:BWR1sh} for~$\wsh$ clearly coincide. Further, condition~\eqref{E:BWR3sh} for~$\wsh$ reads
\begin{align*}
\alpha^{-m}\w(a, b) & + \alpha^{-(m+1)}\w( a \op b, c) + \alpha^{-m}\w(b, c) =\\
& \alpha^{-(m+1)}\w(a \op c, b \op c) + \alpha^{-m}\w(a, c) + \alpha^{-(m+1)}\w(b, c), 
\end{align*}
which, due to the invertibility of~$\alpha$, is the same as~\eqref{E:BWR3multi}.
\end{proof}

We further compare $\alpha$-twisted $\w$-weights with shadow $\wsh$-weights:
\begin{lemma}\label{L:WeightsUnitaryShadow} 
For any $\Q$-colored diagram $(\D,\C)$ and any $\alpha$-twisted quandle $2$-cocycle~$\w$, one has 
\begin{align*}
\BWtw_{\w}(\D,\C,\alpha) &= \BWsh_{\wsh}(\D,\C,\C^{ind}).
\end{align*}
\end{lemma}

\begin{proof}
It follows from the definitions and lemmas above.
\end{proof}

Putting everything together, one gets
\begin{theorem}\label{T:Shadow}
For a knot diagram~$\D$ and an $\alpha$-twisted quandle $2$-cocycle~$\w$, the multiset $\{\,\BWtw_{\w}(\D,\C,\alpha) \,|\, \C \in \Col_{\Q}(\D)\,\}$ of $\alpha$-twisted $\w$-weights of~$\D$ coincides with the multiset
$$\{\,\BWsh_{\wsh}(\D,\C,\Csh) \,|\, (\C,\Csh) \in \Colsh_{\Q,\ZZ}(\D),\, \Csh(\Delta_{ex}(\D)) = 0\,\}$$
of its $\ZZ$-shadow $\wsh$-weights with the exterior region colored by~$0$.
\end{theorem}

\begin{proof}
According to Lemma~\ref{L:WeightsUnitaryShadow}, the first multiset coincides with
$\{\,\BWsh_{\wsh}(\D,\C,\C^{ind}) \,|$ $\C \in \Col_{\Q}(\D)\,\}$. %
 But this is precisely our second multiset, since Lemma~\ref{L:IndColoring} allows to replace any region $\ZZ$-coloring $\Csh$ satisfying $\Csh(\Delta_{ex}(\D)) = 0$ with~$\C^{ind}$. 
\end{proof}

This theorem yields the announced shadow interpretation of twisted quandle cocycle invariants. In particular, combined with Theorem~\ref{T:ShadowInv}, it immediately implies Theorem~\ref{T:Multi}.

\section{Applications}\label{S:Applications}

We now turn to applications of our shadow interpretation of twisted quandle cocycle invariants. Some results on shadow invariants are first recalled. Most of them seem to be folklore; their proofs are included for completeness. Their consequences in the twisted setting are then discussed.

\begin{lemma}\label{L:ShadowCoboundary}
The shadow weights associated to a shadow quandle $2$-coboundary vanish. 
\end{lemma}

The proof we present is well adapted for higher-dimensional generalizations.

\begin{proof}
Take a $\Q$-module $\M$, a map $\theta\colon\M \times \Q \to A$, and a $(\Q,\M)$-colored knot diagram $(\D,\C,\Csh)$. Choose a point~$p$ on the knot, and examine the value $\nu(\gamma) = \theta(\Csh(\Delta(\gamma)),\C(\gamma))$ as you travel along the knot, following its direction; here $\gamma$ is the arc you are on, and $\Delta(\gamma)$ is the region to your right. You pass twice through any crossing~$\x$, and the sum of the increments of~$\nu(\gamma)$ during these two passages is exactly the contribution of~$\x$ to $\BWsh_{d \theta}(\D,\C)$ (for a positive~$\x$, it is shown on Figure~\ref{pic:Proof}; for a negative one things are similar). Returning to~$p$, you recover the original value $\nu(\gamma)$. Considering such walks along each link component, one concludes that the total increment of~$\nu(\gamma)$ --- which is precisely $\BWsh_{d \theta}(\D,\C)$ --- is zero.
\end{proof}

\begin{center}
\psset{unit=1.5mm}
\begin{pspicture}(-10,-2)(75,12)
\psline[linewidth=0.3,arrowsize=1.2,linecolor=MyRed]{->}(0,10)(10,0)
\psline[linewidth=0.3,border=1.4,arrowsize=1.2,linecolor=MyBlue]{->}(10,10)(0,0)
\rput[b](0,11){$\color{MyRed} a$}
\rput[b](10,11){$\color{MyBlue} b$}
\rput[b](-1,-3){$\color{MyBlue} b$}
\rput[b](12,-3){$\color{MyRed} a \lhd b$}
\rput(-3,5){$\scriptstyle\underline{m}$}
\rput(5,10){$\scriptstyle \underline{m \op a}$}
\rput(5,0){$\scriptstyle \underline{m \op b}$}
\rput(15,5){$\scriptstyle \underline{(m \op a) \op b}$}
\rput(50,10){$\color{MyRed} (\theta(m \op b, a \op b) - \theta (m,a))+ $}
\rput(45,5){$\longmapsto\qquad \color{MyBlue} (\theta(m,b) - \theta (m \op a,b)) $}
\rput(50,0){$= d\theta (m,a,b) $}
\end{pspicture}
\psset{unit=1mm}
\captionof{figure}{The total cost of passing through a crossing}\label{pic:Proof}
\end{center}

As a consequence, one obtains

\begin{proposition}
Cohomologous twisted quandle $2$-cocycles yield the same weights and identical twisted quandle cocycle invariants.
\end{proposition}

In particular, one can talk about the weight $\BWtw_{[\w]}$ associated to the twisted quandle $2$-cohomology class of~$\w$.

\begin{proof}
To simplify notations, put $d = d^1_{\opp} - d^1_{\opr}$, $d_{\alpha} = d^1_{\opp} - \alpha d^1_{\opr}$. By the linearity of weights, it suffices to check the triviality of twisted $d_{\alpha}\theta$-weights, with $\theta\colon\Q \to A$. An easy direct verification gives $(d_{\alpha} \theta)_{\alpha} = \alpha d (\theta_{\alpha})$ (we use notation~\eqref{E:AlphaWeight} and its analogue $\theta_{\alpha}(m,a) = \alpha^{-m} \theta (a)$). Lemma~\ref{L:WeightsUnitaryShadow} then implies that
\begin{align*}
\BWtw_{d_{\alpha}\theta}(\D,\C) &= \BWsh_{(d_{\alpha} \theta)_{\alpha}}(\D,\C,\C^{ind}) =  \BWsh_{\alpha d \theta_{\alpha}}(\D,\C,\C^{ind}) = \alpha  \BWsh_{d \theta_{\alpha}}(\D,\C,\C^{ind}),
\end{align*}
which vanishes according to Lemma~\ref{L:ShadowCoboundary}.
\end{proof}

As a second application, we explore an action of a quandle~$\Q$ on $(\Q,\M)$-colorings and its effect on shadow and twisted quandle cocycle invariants.  

\begin{lemma}\label{L:ShadowAction}
The set of $(\Q,\M)$-colorings of a knot diagram~$\D$ has the following $\Q$-module structure: the coloring $(\C,\Csh) \op c$ is obtained from $(\C,\Csh)$ by replacing each employed color~$x$ with $x \op c$. Moreover, for any $\M$-shadow quandle $2$-cocycle~$\w$, the $\w$-weight respects this $\Q$-module structure --- that is, for all $c \in \Q$, one has
\begin{align}\label{E:ShadowAction}
\BWsh_{\w}(\D,\C,\Csh) &= \BWsh_{\w}(\D,(\C,\Csh) \op c).
\end{align}
\end{lemma}

\begin{proof}
We first check that $(\C,\Csh) \op c$ is indeed a $(\Q,\M)$-coloring. For this, one should show that the coloring rules from Figure~\ref{pic:Colorings} remain valid when each color~$x$ is replaced with $x \op c$; this follows from Axiom~\eqref{E:SD} for~$\Q$ and for~$\M$. This axiom also implies that one actually gets a $\Q$-module structure.

To show relation~\eqref{E:ShadowAction}, rewrite the defining property~\eqref{E:BWR3sh} for~$\w$ as 
\begin{gather*}
\w(m, a, b) - \w(m \op c, a \op c, b \op c) =\\
(\w(m \op a, b, c) - \w(m \op b, a \op b, c)) - (\w(m, b, c) - \w(m, a, c)).
\end{gather*}
The second line can be interpreted as $d \w_c(m,a,b)$, where $\w_c \in C^1(\M,\Q,A)$ is defined by $\w_c(m,a) = \w(m, a, c)$, and $d = d^1_{\opp} - d^1_{\opr}$. Consequently, one has
\begin{align*}
\BWsh_{\w}(\D,\C,\Csh) - \BWsh_{\w}(\D,(\C,\Csh) \op c) &= \BWsh_{d \w_c}(\D,\C,\Csh),
\end{align*}
which vanishes according to Lemma~\ref{L:ShadowCoboundary}.
\end{proof}

This lemma leads to a better understanding of the dependence of shadow invariants from Theorem~\ref{T:ShadowInv} on the choice of the exterior region color $m \in \M$. We will use the notion of \emph{orbits} of a $\Q$-module $\M$, which are classes for the equivalence relation on~$M$ induced by $m \sim m \op a$, $a \in \Q$.\footnote{Orbits can also be thought of as the \emph{irreducible components} of~$\M$. Every $\Q$-module uniquely decomposes as a disjoint union of irreducible $\Q$-modules.} 

\begin{proposition}\label{P:ExtColor}
For any $\Q$-module~$\M$, elements~$m_1$ and~$m_2$ lying in the same orbit of~$\M$, and $\M$-shadow quandle $2$-cocycle $\w$ of~$\Q$, the multisets 
\begin{align*}
&\{\,\BWsh_{\w}(\D,\C,\Csh) \,|\,(\C,\Csh) \in \Colsh_{\Q,\M}(\D),\, \Csh(\Delta_{ex}(\D))=m_1\,\}\\
\text{and} \qquad &\{\,\BWsh_{\w}(\D,\C,\Csh) \,|\,(\C,\Csh) \in \Colsh_{\Q,\M}(\D),\, \Csh(\Delta_{ex}(\D))=m_2\,\}
\end{align*}
defining shadow invariants coincide.
\end{proposition}

\begin{proof}
It suffices to check the assertion for $m_2 = m_1 \op c$. In this case, the map $(\C,\Csh) \mapsto (\C,\Csh) \op c$ establishes a bijection between the $(\Q,\M)$-colorings used for defining the two multisets. Lemma~\ref{L:ShadowAction} tells that the $\w$-weights of the corresponding colorings are the same, so the multisets coincide.
\end{proof}

Returning to the twisted setting, one obtains several useful consequences. 

\begin{proposition}\label{P:TwistedScale}
An $\alpha$-twisted quandle cocycle invariant does not change if all its elements are scaled by~$\alpha$.
\end{proposition}

\begin{proof}
The invariants in question are multisets $\{\,\BWtw_{\w}(\D,\C,\alpha) \,|\, \C \in \Col_{\Q}(\D)\,\}$ for $\alpha$-twisted quandle $2$-cocycles $\w$ of~$\Q$. Theorem~\ref{T:Shadow} identifies such a multiset with
\begin{align*}
&\{\,\BWsh_{\wsh}(\D,\C,\Csh) \,|\, (\C,\Csh) \in \Colsh_{\Q,\ZZ}(\D),\, \Csh(\Delta_{ex}(\D)) = 0\,\}.
\end{align*}
According to Proposition~\ref{P:ExtColor}, this multiset does not change if the color $\Csh(\Delta_{ex}(\D))$ of the exterior region of~$\D$ is imposed to be $-1$ instead of $0$, since the $\Q$-module $\ZZ$ consists of a single orbit. With this new condition, all region colors in $\Csh$ should be reduced by $1$. Recalling the definition~\eqref{E:AlphaWeight} of the shadow quandle 2-cocycle
$\wsh$, one sees that this corresponds to multiplying all the shadow $\wsh$-weights by~$\alpha$. Thus such a scaling does not change our invariant.
\end{proof}

This proposition can be used to show that certain twisted quandle cocycle invariants are no stronger than usual quandle invariants. For instance, for a finite quandle~$\Q$ and $A = \ZZ[t^{\pm 1}]$, $t$-twisted quandle cocycle invariants contain only zeros, since a non-zero weight would imply, by scaling by~$t$, an infinity of pairwise distinct weights, whereas the total number of $\Q$-colorings is finite. 

\begin{proposition}\label{P:TwistedScale2}
Suppose that for $c_1,\ldots, c_k \in \Q$ and $\varepsilon_1,\ldots,\varepsilon_k \in \{\pm 1\}$, relation
\begin{align}\label{E:RelInQ}
(\cdots (a \op^{\varepsilon_1} c_1) \op^{\varepsilon_2} \cdots ) \op^{\varepsilon_k} c_k &=a
\end{align}
holds for all $a \in \Q$; we used notations $\op^{+1} = \op$, $\op^{-1} = \wop$. Then the multiplication by $\alpha^{\sum_{i=1}^k \varepsilon_i}-1$ annihilates all $\alpha$-twisted weights.
\end{proposition}

\begin{proof}
As usual, Lemma~\ref{L:WeightsUnitaryShadow} allows us to work with shadow weights instead of $\alpha$-twisted weights. A repeated application of Lemma~\ref{L:ShadowAction} gives
\begin{align*}
\BWsh_{\wsh}(\D,\C,\C^{ind}) &= \BWsh_{\wsh}(\D,(\cdots ((\C,\C^{ind}) \op^{\varepsilon_1} c_1) \op^{\varepsilon_2} \cdots ) \op^{\varepsilon_k} c_k).
\end{align*}
Relation~\eqref{E:RelInQ} implies $(\cdots ((\C,\C^{ind}) \op^{\varepsilon_1} c_1) \op^{\varepsilon_2} \cdots ) \op^{\varepsilon_k} c_k = (\C,\C^{ind} + \sum_{i=1}^k \varepsilon_i)$, where in the latter region coloring each color is increased by $\sum_{i=1}^k \varepsilon_i$. Repeating the argument from the proof of Proposition~\ref{P:TwistedScale}, one obtains
\begin{align*}
\textstyle \BWsh_{\wsh}(\D,\C,\C^{ind} + \sum_{i=1}^k \varepsilon_i) &= \alpha^{-\sum_{i=1}^k \varepsilon_i}\BWsh_{\wsh}(\D,\C,\C^{ind}).
\end{align*}
Now, combine everything to get $\BWtw_{\w}(\D,\C,\alpha) = \alpha^{-\sum_{i=1}^k \varepsilon_i}\BWtw_{\w}(\D,\C,\alpha)$.
\end{proof}

For example, if $\Q$ contains a ``central'' element~$c$ satisfying $a \op c = a$ for all $a \in \Q$, then scaling by $\alpha - 1$ kills all $\alpha$-twisted weights. 

The third application is a shadow enhancement of twisted quandle cocycle invariants. To treat simultaneously the $\Q$-module~$\M$ we want to color regions with and the $\Q$-module $\ZZ$ (with $m \op a = m+1$) used in our shadow interpretation of twisted invariants, we use the following elementary observation:

\begin{lemma}\label{L:ProductQMod}
The direct product $\M \times \M'$ of $\Q$-modules $\M$ and $\M'$ is also a $\Q$-module, with the diagonal action $(m,m') \op a = (m \op a, m' \op a)$. 
\end{lemma}

Now, repeating verbatim all the arguments from the previous and the beginning of the current sections for the $\Q$-module $\M \times \ZZ$ instead of~$\ZZ$, one gets

\begin{theorem}\label{T:ShadowTwisted}
Take a $\Q$-module $\M$, a fixed $m \in \M$, and an $\alpha$-twisted $\M$-shadow quandle $2$-cocycle~$\w$. The multiset of \emph{$\alpha$-twisted $\M$-shadow $\w$-weights} 
\psset{unit=0.4mm}
\begin{align*}
&\{\,\BWtwsh_{\w}(\D,\C,\Csh,\alpha) \,|\, (\C,\Csh)\in \Colsh_{\Q,\M}(\D),\,\Csh(\Delta_{ex}(\D))=m \,\}, \\
\text{where} & \qquad\BWtwsh_{\w}(\D,\C,\Csh,\alpha) = \sum_{\begin{pspicture}(13,10)(-3,0)
\psline(0,10)(10,0)
\psline[border=0.8,arrowsize=3]{->}(10,10)(0,0)
\rput(-2,10){$\scriptstyle a$}
\rput(12,10){$\scriptstyle b$}
\end{pspicture}} \sgn(\x) \alpha^{-\ind(\x)} \w(\Csh(\Delta(\x)), a, b),
\end{align*}
\psset{unit=1mm}
is a knot invariant\footnote{For notations, refer to Definition~\ref{D:IndexRegion}.}.  Moreover, cohomologous $2$-cocycles yield identical invariants.
\end{theorem}

The last bonus from our shadow interpretation of twisted quandle cocycle invariants is their higher-dimensional generalization. Indeed, take a $\Q$-colored $k-1$-dimensional knot diagram $(\D,\C)$ in~$\RR^k$, and an $\alpha$-twisted quandle $k$-cocycle~$\w$. Define the $\alpha$-twisted $\w$-weight of $(\D,\C)$ by a formula analogous to~\eqref{E:TotalWeightMulti}, using higher-dimensional versions of indices and signs for crossings of multiplicity~$k$, and a relevant ordering of adjacent sheets. Lemmas \ref{L:IndColoring}-\ref{L:WeightsUnitaryShadow} 
generalize directly. 
The multiset of $\alpha$-twisted $\w$-weights is thus interpreted in terms of higher-dimensional shadow invariants (for more details about the latter, see for instance \cite{FR,CKS_Surfaces,PrzRos}). Moreover, this construction admits a shadow version, similar to that from Theorem~\ref{T:ShadowTwisted}. Hence any $\alpha$-twisted (possibly shadow) quandle $k$-cocycle 
gives rise to a $k-1$-dimensional knot invariant. 

\section{Positive quandle cocycle invariants via shadows}\label{S:PosShadow}

This section is devoted to the case $\alpha = -1$. For this~$\alpha$, the twisted quandle cohomology theory is precisely the positive quandle cohomology theory from~\cite{PosQuandleHom}. In order to show that positive and $-1$-twisted quandle cocycle invariants coincide, it remains to compare the weights used in the two constructions:

\begin{lemma}\label{L:WeightsPosTwisted} 
For any $\Q$-colored diagram $(\D,\C)$ and any positive quandle $2$-cocycle~$\w$, one has 
$\BWtw_{\w}(\D,\C,-1) = \BWpos_{\w}(\D,\C)$.
\end{lemma}

\begin{proof}
Is is sufficient to show that for any crossing~$\x$, the signs $\sgn(\x) (-1)^{-\ind(\x)}$ and $\sgnpos(\x)$ 
coincide. First, observe that, modulo $2$, the coloring~$\C^{ind}$ reduces to the checkerboard coloring, under the identification $0 = white$, $1 = black$. Using this identification, the definition of $\sgnpos(\x)$ (Figure~\ref{pic:ColoringsPos}) can be restated as on Figure~\ref{pic:ColoringsPos2} (note that we turned both diagrams so that the over-arc points south-west). Now, if the under-arc points south-east, then $\sgn(\x)$ is~$1$, and $\ind(\x)$ is the color of the region to the west of~$\x$; one checks that the latter satisfies $(-1)^{-\ind(\x)} = \sgnpos(\x)$, as desired. A change in the orientation of the under-arc changes $\sgn(\x)$ and $\ind(\x)$, but not $\sgnpos(\x)$, so equality $\sgn(\x) (-1)^{-\ind(\x)} = \sgnpos(\x)$ is preserved.
\end{proof}

\begin{center}
\begin{pspicture}(0,0)(40,10)
\psline[linewidth=0.6](0,10)(10,0)
\psline[linewidth=0.6,border=1.8,arrowsize=2]{->}(10,10)(0,0)
\rput(0,5){$0$}
\rput(10,5){$0$}
\rput(5,0){$1$}
\rput(5,10){$1$}
\rput(20,5){$\leadsto \color{MyBlue} \bm{+}$}
\end{pspicture}
\begin{pspicture}(0,0)(30,10)
\psline[linewidth=0.6](0,10)(10,0)
\psline[linewidth=0.6,border=1.8,arrowsize=2]{->}(10,10)(0,0)
\rput(0,5){$1$}
\rput(10,5){$1$}
\rput(5,0){$0$}
\rput(5,10){$0$}
\rput(20,5){$\leadsto \color{MyBlue} \bm{-}$}
\end{pspicture}
\captionof{figure}{Sign $\sgnpos$ via region indices}
\label{pic:ColoringsPos2}
\end{center}

Thus positive quandle cocycle invariants coincide with $-1$-twisted ones. Combining this with Theorem~\ref{T:Shadow}, one obtains
\begin{theorem}\label{T:Pos}
For a knot diagram~$\D$ and a positive quandle $2$-cocycle~$\w$, the multiset $\{\,\BWpos_{\w}(\D,\C) \,|\, \C \in \Col_{\Q}(\D)\,\}$ of positive $\w$-weights of~$\D$ coincides with the multiset
$$\{\,\BWsh_{\w_{-1}}(\D,\C,\Csh) \,|\, (\C,\Csh) \in \Colsh_{\Q,\ZZ_2}(\D),\, \Csh(\Delta_{ex}(\D)) = 0\,\}$$
of its $\ZZ_2$-shadow $\w_{-1}$-weights with the exterior region colored by~$0$.
\end{theorem}

Note that we replaced our $\Q$-module $\ZZ$ with~$\ZZ_2$, with the same module operation $m \op a = m+1$; this is possible since $(-1)^m$ is completely determined by $m \operatorname{mod} 2$.

The theorem interprets positive quandle cocycle invariants as shadow ones. This directly shows their invariance, and gives for free their shadow enhancement, answering  a question raised in~\cite{PosQuandleHom}. Moreover, this implies analogs of Propositions~\ref{P:TwistedScale} and~\ref{P:TwistedScale2} in the positive setting:

\begin{proposition}\label{P:PosScale}
A positive quandle cocycle invariant is symmetric with respect to~$0$ (that is, it does not change if all its elements change signs).
\end{proposition}

\begin{proposition}\label{P:PosScale2}
Suppose that for some odd~$k$ and some $c_1,\ldots, c_k \in \Q$ and $\varepsilon_1,\ldots,\varepsilon_k \in \{\pm 1\}$, relation~\eqref{E:RelInQ} holds for all $a \in \Q$. Then all non-trivial positive weights are of order~$2$ in the abelian group~$A$.
\end{proposition}

\section{Twisted quandle cocycle invariants for links}\label{S:TwistedLinks}

For links with several components, twisted quandle cocycle invariants can be sharpened using a method developed in this section. In what follows, the diagram~$D$ represents a link with $k$ ordered components, numbered $1,\ldots,k$.

First, we need a more subtle notion of index. 
\begin{definition}\label{D:IndexRegionLink}
\begin{itemize}
\item The \emph{$j$th index} $\ind_j(\Delta)$ of a region~$\Delta$ in~$\D$ is the index of~$\Delta$ with respect to the $j$th component of~$\D$, in the sense of Definition~\ref{D:IndexRegion}.
\item The \emph{$j$th index} of a crossing~$\x$ is set to be $\ind_j(\x) = \ind_j(\Delta(\x))$. 
 \end{itemize}
\end{definition}

For a quandle $(\Q,\op)$, consider the equivalence relation induced by $a \sim a \op b$. Corresponding equivalence classes are called \emph{orbits}, their set is denoted by $\Orb(\Q)$, and the orbit of an $a \in \Q$ is denoted by $\O(a)$. A $\Q$-coloring $\C$ of~$\D$ assigns elements of the same orbit to all the arcs belonging to the same component of~$\D$; for the $j$th component, this orbit is denoted by $\Corb(j)$.

As before, let~$A$ be a module over a commutative ring~$R$. Take a collection $\oalpha = (\alpha_\O \in R^*)_{\O \in \Orb(\Q)}$.

\begin{definition}
\begin{itemize}
\item An \emph{$\oalpha$-twisted quandle $2$-cocycle} is a map $\w\colon\Q \times \Q \to A$ satisfying \eqref{E:BWR1} and
\begin{align}
\alpha_{\O(c)}^{-1}\w(a \op c, b \op c) -  \w(a, b) - \alpha_{\O(b)}^{-1}\w&(a \op b, c) + \w(a, c) \notag\\
&+ (\alpha_{\O(a)}^{-1} - 1)\w(b, c) =0.\label{E:BWR3Link}
\end{align}
\item The \emph{($\oalpha$-twisted) $\w$-weight} of a $\Q$-colored link diagram $(\D,\C)$ is the sum
\psset{unit=0.4mm}
\begin{align}
\BWtw_{\w}(\D,\C,\oalpha) &= \sum_{\begin{pspicture}(13,10)(-3,0)
\psline(0,10)(10,0)
\psline[border=0.8,arrowsize=3]{->}(10,10)(0,0)
\rput(-2,10){$\scriptstyle a$}
\rput(12,10){$\scriptstyle b$}
\end{pspicture}} \sgn(\x)\, \alpha_{\Corb(1)}^{-\ind_1(\xx)}\cdots \alpha_{\Corb(k)}^{-\ind_k(\xx)}\, \w(a, b).\label{E:TotalWeightLink}
\end{align}
\psset{unit=1mm}
\end{itemize}
\end{definition} 

\begin{theorem}\label{T:Link}
The multiset $\{\,\BWtw_{\w}(\D,\C,\oalpha) \,|\, \C \in \Col_{\Q}(\D)\,\}$ depends only on the underlying link (and not on the choice of its diagram~$\D$), and only on the $\oalpha$-twisted quandle cohomology class $[\w]$ of~$\w$. 
\end{theorem}

The theorem is proved by interpreting the constructed multisets as shadow invariants. The $\Q$-module one should use here is the following one:

\begin{lemma}
The set $\displaystyle \bigoplus_{\O \in \Orb(\Q)}\ZZ e_\O$ with $m \op a = m + e_{\O(a)}$ is a $\Q$-module.
\end{lemma}

\begin{proof}
One has to check that $e_{\O(b)} + e_{\O(c)} = e_{\O(c)} + e_{\O(b \op c)}$, which is obvious since, by definition, $b$ and $b \op c$ belong to the same orbit.
\end{proof}

As usual, our shadow interpretation gives for free shadow and higher-dimensional generalizations of the theorem.

\begin{remark}
Following~\cite{Inoue}, the multiset from the theorem can be decomposed into multisets $\{\,\BWtw_{\w}(\D,\C,\oalpha) \,|\, \C \in \Col_{\Q}(\D), \Corb(1)=\O_1, \ldots,\Corb(k)=\O_k\,\}$, indexed by $k$-tuples of orbits $\O_1,\ldots, \O_k \in \Orb(\Q)$. The invariance of these smaller multisets follows from the fact that Reidemeister moves and induced local coloring changes preserve the orbit of the colors assigned to the arcs of a given link component.
\end{remark}

{\bfseries Acknowledgements.} The authors are grateful to the organizers of \emph{Knots and Low Dimensional Manifolds} (a satellite conference of Seoul ICM 2014), where the idea of this paper was born. They would like to thank Zhiyun Cheng for interesting discussions. This work was made possible thanks to JSPS KAKENHI Grants 25$\cdot$03315, 26287013, and 26400082. V.L. was also supported by a JSPS Postdoctral Fellowship for Foreign Researchers. 

\bibliographystyle{alpha}
\bibliography{biblio}

\newcommand{\etalchar}[1]{$^{#1}$}
\begin{thebibliography}{CJK{\etalchar{+}}03}

\bibitem[Ale23]{AlexanderNb}
J.W. Alexander.
\newblock A lemma on systems of knotted curves.
\newblock {\em Proc. Nat. Acad. Science USA}, 9:93--95, 1923.

\bibitem[CES02]{TwistedQuandle}
J.S. Carter, M. Elhamdadi, and M. Saito.
\newblock Twisted quandle homology theory and cocycle knot invariants.
\newblock {\em Algebr. Geom. Topol.}, 2:95--135 (electronic), 2002.

\bibitem[CG14]{PosQuandleHom}
Z.~{Cheng} and H.~{Gao}.
\newblock {Positive quandle homology and its applications in knot theory}.
\newblock {\em ArXiv e-prints}, March 2014.

\bibitem[CJK{\etalchar{+}}03]{QuandleHom}
J.S. Carter, D. Jelsovsky, S. Kamada, L. Langford, and M. Saito.
\newblock Quandle cohomology and state-sum invariants of knotted curves and
  surfaces.
\newblock {\em Trans. Amer. Math. Soc.}, 355(10):3947--3989, 2003.

\bibitem[CKS00]{CKS_AlexanderNb}
J.S. Carter, S. Kamada, and M. Saito.
\newblock Alexander numbering of knotted surface diagrams.
\newblock {\em Proc. Amer. Math. Soc.}, 128(12):3761--3771, 2000.

\bibitem[CKS01]{CKS_Geometric}
J.S. Carter, S. Kamada, and M. Saito.
\newblock Geometric interpretations of quandle homology.
\newblock {\em J. Knot Theory Ramifications}, 10:345--386, 2001.

\bibitem[CKS04]{CKS_Surfaces}
J.~S. Carter, S. Kamada, and M. Saito.
\newblock {\em Surfaces in 4-space}, volume 142 of {\em Encyclopaedia of
  Mathematical Sciences}.
\newblock Springer-Verlag, Berlin, 2004.
\newblock Low-Dimensional Topology, III.

\bibitem[CN11]{ChangNelson}
W. Chang and S. Nelson.
\newblock Rack shadows and their invariants.
\newblock {\em J. Knot Theory Ramifications}, 20(9):1259--1269, 2011.

\bibitem[CS98]{CS_surfaces}
J.~S. Carter and M. Saito.
\newblock {\em Knotted surfaces and their diagrams}, volume~55 of {\em
  Mathematical Surveys and Monographs}.
\newblock American Mathematical Society, Providence, RI, 1998.

\bibitem[FR92]{FR}
R. Fenn and C. Rourke.
\newblock Racks and links in codimension two.
\newblock {\em J. Knot Theory Ramifications}, 1(4):343--406, 1992.

\bibitem[FRS95]{RackHom}
R. Fenn, C. Rourke, and B. Sanderson.
\newblock Trunks and classifying spaces.
\newblock {\em Appl. Categ. Structures}, 3(4):321--356, 1995.

\bibitem[Ino13]{Inoue}
A. Inoue.
\newblock Quasi-triviality of quandles for link-homotopy.
\newblock {\em J. Knot Theory Ramifications}, 22(6):1350026, 10, 2013.

\bibitem[Joy82]{Joyce}
D. Joyce.
\newblock A classifying invariant of knots, the knot quandle.
\newblock {\em J. Pure Appl. Algebra}, 23(1):37--65, 1982.

\bibitem[Kam02]{Kamada}
S. Kamada.
\newblock Knot invariants derived from quandles and racks.
\newblock In {\em Invariants of knots and 3-manifolds ({K}yoto, 2001)},
  volume~4 of {\em Geom. Topol. Monogr.}, pages 103--117 (electronic). Geom.
  Topol. Publ., Coventry, 2002.

\bibitem[Leb13]{Lebed1}
V. Lebed.
\newblock Homologies of algebraic structures via braidings and quantum
  shuffles.
\newblock {\em J. Algebra}, 391:152--192, 2013.

\bibitem[Mat82]{Matveev}
S.V. Matveev.
\newblock Distributive groupoids in knot theory.
\newblock {\em Mat. Sb. (N.S.)}, 119(161)(1):78--88, 160, 1982.

\bibitem[PR13]{PrzRos}
J.H. {Przytycki} and W. {Rosicki}.
\newblock {Cocycle invariants of codimension 2-embeddings of manifolds}.
\newblock {\em ArXiv e-prints}, October 2013.

\bibitem[Prz11]{Prz1}
J.H. Przytycki.
\newblock Distributivity versus associativity in the homology theory of
  algebraic structures.
\newblock {\em Demonstratio Math.}, 44(4):823--869, 2011.

\bibitem[PS14]{PrzSikora}
J.H. Przytycki and A.S. Sikora.
\newblock Distributive products and their homology.
\newblock {\em Comm. Algebra}, 42(3):1258--1269, 2014.

\bibitem[RS00]{RS_Links}
C.~{Rourke} and B.~{Sanderson}.
\newblock {A new classification of links and some calculations using it}.
\newblock {\em ArXiv Mathematics e-prints}, June 2000.

\end{thebibliography}

\end{document}